\documentclass[11pt, twoside]{article}
\usepackage{amsfonts,amssymb,amsmath,amsthm}
\usepackage{authblk}
\usepackage{graphicx}
\usepackage[all]{xy}
\usepackage{tikz}
\usepackage{changepage}
\usepackage{psfrag,xmpmulti,amscd,color,pstricks, import,enumerate,soul}
\usepackage[normalem]{ulem}

 \divide\oddsidemargin by 2
\newtheorem{thm}{Theorem}[section]
\newtheorem{cor}[thm]{Corollary}
\newtheorem{lem}[thm]{Lemma}

\newtheorem{summ}[thm]{Summary}
\newtheorem*{thmA}{Theorem A}
\newtheorem*{thmB}{Theorem B}

\newtheorem*{thmC}{Theorem C}
\newtheorem*{thmD}{Theorem D}
\newtheorem*{thmE}{Theorem E}
\newtheorem*{thmF}{Theorem F}

\theoremstyle{definition}
\newtheorem{defn}[thm]{Definition}

\newtheorem{rem}[thm]{Remark}
\newtheorem*{rem*}{Remark}
\newtheorem*{rems*}{Remarks}
\newtheorem{ex}[thm]{Example}
\newtheorem*{ex*}{Example}

\numberwithin{equation}{section}

\definecolor{OrangeRed}{cmyk}{0,0.6,1,0}            
\definecolor{DarkBlue}{cmyk}{1,1,0,0.20}
\definecolor{DarkGreen}{cmyk}{1,0,0.6,0.2}
\definecolor{myblue}{rgb}{0.66,0.78,1.00}
\definecolor{Violet}{cmyk}{0.79,0.88,0,0}
\definecolor{Lavender}{cmyk}{0,0.48,0,0}
\definecolor{purpleheart}{rgb}{0.41, 0.21, 0.61}
\definecolor{brick}{cmyk}{0,0.8,0.3,0.5}
\definecolor{LightGray}{cmyk}{0,0,0,0.5}

\newcommand{\dist}{\operatorname{dist}}

\newcommand{\D}{{\mathbb D}}

\newcommand{\N}{{\mathbb N}}

\renewcommand{\epsilon}{\varepsilon}
\renewcommand{\phi}{\varphi}

\renewcommand{\tilde}{\widetilde}

\vspace{5cm}

\title{Shrinking targets and recurrent behaviour for forward compositions of inner functions}

\author[1]{\hspace{1.7cm}Anna Miriam Benini \thanks{Partially supported by GNAMPA, INdAM; by the French Italian University and Campus France through the Galileo program, under the project “From rational to transcendental: complex dynamics and parameter spaces”; PRIN (2022) Real and Complex Manifolds: Geometry and holomorphic dynamics}}
\author[4]{Vasiliki Evdoridou}
\author[2,3]{N\'uria Fagella 
 \thanks{Supported by the Spanish State Research Agency through the MdM grant CEX2020-001084-M and grant
PID2020-118281GB-C32; and by the Catalan government through ICREA Academia 2020. }}
\author[4]{\hspace{2cm} Philip J. Rippon}
\author[4]{Gwyneth M. Stallard}
\affil[1]{\small Dep. of Mathematical, Physical and Computer Sciences, Universit\`a di Parma, Italy.}
\affil[2]{\small Dep. de Matem\`atiques i Inform\`atica, Universitat de Barcelona, Catalonia, Spain.}
\affil[3]{\small Centre de Recerca Matemàtica, Bellaterra, Catalonia, Spain.}
\affil[4]{\small School of Mathematics and Statistics, The Open University, Milton Keynes, UK.}
\date{\today}

\date{}
\begin{document}

\maketitle
\begin{abstract}

 We prove sharp results about recurrent behaviour of orbits of forward compositions of inner functions, inspired by fundamental results about iterates of inner functions, and give examples to illustrate behaviours that cannot occur in the simpler case of iteration.

A result of Fern\'andez, Meli\'an and Pestana gives a  precise version of the classical Poincar\'e recurrence theorem for iterates of the boundary extension of an inner function that fixes~0. We generalise this to forward composition sequences $F_n=f_n\circ \dots\circ f_1,$ $n\in \N,$ where $f_n$ are inner functions that fix~0, giving conditions on the contraction of $(F_n)$ so that the radial boundary extension $F_n$ hits any shrinking target of arcs $(I_n)$ of a given size.

Next, Aaronson, and also Doering and Ma\~n\'e, gave a remarkable dichotomy for iterates of any inner function, showing that the behaviour of the boundary extension is of two entirely different types, depending on the size of the sequence $(|f^n(0)|)$. In earlier work, we showed that one part of this dichotomy holds in the non-autonomous setting  of forward compositions.

It turns out that this dichotomy is closely related to the result of Fern\'andez, Meli\'an and Pestana, and here we show that a version of the second part of the dichotomy holds in the non-autonomous setting provided we impose a condition on the contraction of $(F_n)$ in relation to the size of the sequence $(|F_n(0)|)$. The techniques we use include a strong version of the second Borel--Cantelli lemma and  strong mixing results of Pommerenke for contracting sequences of inner functions. We give examples to show that the contraction conditions that we need to impose in the non-autonomous setting are best possible.
\end{abstract}

\section{Introduction}
Classical ergodic theory concerns the behaviour of the iterates $T^n$, $n\in \N$, of a transformation~$T$ on a space~$X$, which is measure-preserving in the sense that the measure of $T^{-1}(E)$ is identical to that of $E$ for all measurable sets $E\subset X$, with a focus on properties such as recurrence, mixing, ergodicity, exactness, etc. Less research has so far been carried out on ergodic theory in the {\em non-autonomous} setting when compositions of sequences of measure-preserving transformations are considered instead of iterates of one particular transformation. In this paper we focus on ergodic theory for such compositions, in particular for forward compositions of radial boundary extensions of inner functions; an inner function $f$ is a holomorphic self-map of the unit disc $\D=\{z:|z|<1\}$ whose radial boundary extension, which we also denote by~$f$, maps $\partial \D$ to $\partial \D$, apart from a set of measure 0.

A foundational result in ergodic theory is the Poincar\'{e} recurrence theorem, a strong consequence of which can be stated as follows; see \cite[Theorem~A'] {Fern}, for example.

\begin{thm}\label{PRT1}
Let $(X,d)$ be a separable metric space, $\mu$ a positive measure on $X$ normalised so $\mu(X)=1$, and $T:X \to X$ a measure-preserving transformation which is also ergodic (that is, for every measurable $E \subset  X$ with $E = T^{-1}(E)$ we have
$\mu(E) = 0$ or $\mu(X \setminus E) = 0$).

Then, for every $x_0\in X$,
\begin{equation}\label{ThmA'}
\liminf_{n\to\infty}d(T^n(x),x_0)= 0, \; \text{ for } \mu\text{-almost every } x\in X.
\end{equation}
In particular, almost every orbit under $T$ is dense in $X$.
\end{thm}

If $T$ is the boundary extension of an inner function $f$ such that $f(0)=0$ then $T$ is (Lebesgue) measure-preserving; see Lemma~\ref{lem:Lowner}.  Following \cite{ivrii-urb}, we say that an inner function~$f$ satisfying $f(0)=0$ is {\em centred}.

In this centred case, and assuming that $f$ is not a rotation, Fern\'andez et al \cite[Theorem~2] {Fern} proved the following version of Theorem~\ref{PRT1} with a stronger conclusion. Here, and subsequently, we denote the radial boundary extension of an inner function $f$ by the same letter.
\begin{thm}\label{FernThm2}
Let $f$ be a centred inner function such that $|f'(0)|<1$. Then, for any positive decreasing sequence $(r_n)$ such that $\sum_{n=1}^\infty r_n=\infty$ and every $\zeta_0\in\partial \D$, we have
\begin{equation}\label{target}
\liminf_{n\to\infty}|f^n(\zeta)-\zeta_0|/r_n\le 1,\quad \text{for almost every } \zeta\in \partial \D.
\end{equation}

\end{thm}
The proof uses the fact that in this situation the radial boundary values of $f$ satisfy a certain strong mixing property due to Pommerenke, which we describe in Section~\ref{Prelims}.

Estimates such as \eqref{target} can be thought of as examples of `shrinking target' results, where we study the set of points whose orbits hit a shrinking sequence of sets infinitely often, a concept introduced by Hill and Velani \cite[Theorems~1--4]{HV} in the context of the behaviour of orbits lying in the Julia set of an expanding rational map. Since then, the notion of shrinking targets has played a major role in ergodic theory and its applications; see, for example, \cite{MagnusA},  \cite{BakerKoi}, \cite{Kirsebom}.

 There are two main aims of this paper. The first is to identify conditions (ideally best possible) under which shrinking target results of this type hold for non-autonomous systems of forward composition sequences $F_n=f_n\circ \cdots \circ f_1$, where each $f_n, n\in \N$ is a centred inner function. It is useful here to adopt a rather general definition of shrinking target.

\begin{defn}
A {\em target} is a sequence $(I_n)$ of arcs of $\partial \D$, not necessarily nested and not necessarily shrinking in length. The target is {\em shrinking} if $|I_n| \to 0$ as $n\to \infty$,  where $|I_n|$  denotes the length of the arc $I_n$.

Let $F_n$ be inner functions for $n\in \N$ and let $\zeta\in\partial\D$. The sequence $(F_n(\zeta))$ {\em hits} the target $(I_n)$ if we have $F_n(\zeta)\in I_n$ infinitely often; that is,
\[
\zeta \in \limsup_{n\to\infty}F_n^{-1}(I_n):=\bigcap_{N=1}^\infty \bigcup_{n=N}^\infty F_n^{-1}(I_n).
\]
\end{defn}
 We start with the following theorem, which is easy to prove,  in which the arcs in the shrinking target have finite total length; note that the functions $F_n$ in this result are not necessarily forward compositions and also not necessarily inner functions.
\begin{thmA}\label{thm:easy}
Let $(F_n)$  be a sequence of holomorphic self-maps of $\mathbb{D}$ such that $F_n(0)=0$ for $n\in \N$ and suppose that $(I_n)$ is a shrinking target in $\partial\D$ such that
\[
\sum_{n=1}^\infty |I_n|<\infty.
\]
Then $(F_n(\zeta))$ fails to hit $(I_n)$ for almost all $\zeta\in\partial\D$.
\end{thmA}
The case of shrinking targets $(I_n)$ for which $\sum_{n=1}^\infty |I_n|$ is divergent, corresponding to the condition $\sum_{n=1}^\infty r_n=\infty$ in Theorem~\ref{FernThm2}, is more interesting and we will focus on this condition in the context of forward compositions of centred inner functions.

 Before doing so, we mention the other main aim of our work, which arises from a striking dichotomy obtained by Aaronson~\cite{aaronson} and by Doering and Ma\~{n}\'{e}~\cite{doering-mane} showing that iterates of the boundary values of a general inner function $f$ (without the restriction that $f(0)=0$) behave in two entirely different ways, depending on the proximity of the orbits of interior points to $\partial \D$. In \cite{BEFRS2} we made progress on extending this dichotomy to the non-autonomous setting and here we use our results on shrinking targets to obtain significant new results which show that in the non-autonomous case there are interesting differences from the case of iteration of a single inner function.

In \cite{BEFRS2} we introduced a classification of forward compositions of holomorphic maps based on the extent to which these sequences contract the hyperbolic metric;  see also~\cite{BEFRS}, where we introduced this classification in the context of wandering domains of transcendental entire functions. In particular, we say that such a forward composition sequence $(F_n)$ is {\em contracting} if the hyperbolic distance between $F_n(z)$ and $F_n(z')$ tends to 0 as $n\to\infty$ for any  (and therefore every) distinct pair $z,z'\in\D$. Examples in \cite[Section~8]{BEFRS2} show that for forward compositions of inner functions we can only hope to  obtain a dichotomy of the type given in~\cite{aaronson} and~\cite{doering-mane} by assuming that these forward compositions are contracting.

In the situation where $F_n=f_n\circ \dots\circ f_1$ is a forward composition of self-maps of $\D$, the sequence $(F_n)$ is contracting if and only if
\begin{equation}\label{eq:contracting}
\lambda_n \ldots \lambda_1 \to 0\;\text{ as } n\to \infty,\quad\text{or equivalently,}\quad \sum_{n=1}^\infty (1-\lambda_n) = \infty,
\end{equation}
where $\lambda_n$ is the {\em hyperbolic distortion} of $f_n$ at $F_{n-1}(0)$, given by
\begin{equation}\label{lambda-n}
\lambda_n=\frac{\rho_\D(F_{n}(0)) |f_n'(F_{n-1}(0))|}{\rho_\D(F_{n-1}(0))};
\end{equation}
see \cite[Theorem~7.2]{BEFRS2}. Here $\rho_\D(z)$ denotes the hyperbolic density at $z \in \D$.  In this situation, we always take $F_0$ to be the identity map. In the case when the functions $f_n$ fix 0, we have $\lambda_n=|f'_n(0)|$.

We now return to the question of whether shrinking targets are hit by forward compositions of centred inner functions. We have the following basic dichotomy, which follows from another result of Pommerenke given in Section~\ref{Prelims}.
\begin{thmB}\label{dich}
Let $F_n=f_n\circ \cdots \circ f_1$, $n\in \N$, be a contracting forward composition of centred inner functions and let $(I_n)$ be any shrinking target. Then
\[
\{\zeta\in\partial \D: (F_n(\zeta)) \text{ hits } (I_n)\}
\]
has either full measure or zero measure in $\partial \D$.
\end{thmB}
We give examples in Section~\ref{Examples} which show that this dichotomy does not hold if we omit the contracting hypothesis.

Theorem~B tells us that whenever $(I_n)$ is a shrinking target and $(F_n)$ is a contracting forward composition of centred inner functions, then either $(F_n(\zeta))$ hits $(I_n)$ almost always or $(F_n(\zeta))$ hits $(I_n)$ almost never, and Theorem~A tells us that the latter is the case whenever $\sum_{n=1}^\infty |I_n|<\infty$.

Therefore, when studying shrinking targets for which $\sum_{n=1}^\infty |I_n|=\infty$, it makes sense to deal with forward compositions of centred inner functions that are contracting. We start with the {\em uniformly contracting} case where there is a uniform bound on $|f'_n(0)|$, as is the case for the iteration of a single inner function in Theorem~\ref{FernThm2}. We prove that such forward compositions almost always hit {\em any} shrinking target $(I_n)$ such that $\sum_{n=1}^\infty |I_n|=\infty$. This is a generalisation of Theorem~\ref{FernThm2}.
\begin{thmC}\label{thm:uniformcontr}
Let $f_n$ be centred inner functions with $|f_n'(0)| \leq \lambda<1$ for $n \in \N$, and let $F_n = f_n \circ \cdots \circ f_1$. Then, for any shrinking target $(I_n)$ in $\partial \D$ that satisfies
\begin{equation}\label{Indiv}
\sum _{n=1}^{\infty} |I_n|= \infty,
\end{equation}
the sequence $(F_n(\zeta))$ hits $(I_n)$ for almost every $\zeta\in\partial \D$.
\end{thmC}
The uniform bound on $|f'_n(0)|$ in the hypothesis of Theorem~C allows us to give a similar proof to that of Theorem~\ref{FernThm2} in \cite{Fern}. Both proofs rely on the strong mixing result of Pommerenke,  mentioned earlier, though in our proof we appeal to a version of the second Borel--Cantelli lemma due to Philipp rather than the Paley--Zygmund inequality used in \cite{Fern}.

Theorem~C is best possible, in the sense that if $\lambda_n\to 1^-$ as $n\to\infty$  in such a way that $\lambda_n\ldots \lambda_1 \to 0$ as $n\to \infty$, then there is a shrinking target $(I_n)$ that satisfies \eqref{Indiv} and a (contracting) composition of centred inner functions $(f_n)$ such that $|f'_n(0)|=\lambda_n$ for $n\in \N$, and $F_n=f_n\circ\cdots\circ f_1$ fails to hit $(I_n)$; see Example~\ref{ex015}.  In particular, it is not the case that any contracting composition of centred inner functions must almost always hit any shrinking target for which $\sum _{n=1}^{\infty} |I_n|= \infty$.

 The following more general result suggests that the smaller a shrinking target is, with $\sum_{n=1}^\infty |I_n|=\infty$, the greater the amount of contraction that is required to hit the target.

\begin{thmD} \label{thm:contracting}
For $n\in \N$, let $f_n$ be centred inner functions $|f_n'(0)| =\lambda_n=1-\mu _n$, and let $F_n = f_n \circ \cdots \circ f_1$. Then, for any shrinking target $(I_n)$ in $\partial \D$ that satisfies
\begin{equation}\label{eq:muI1}
\sum_{n=1}^\infty \mu_n|I_n|=\infty,
\end{equation}
the sequence $(F_n(\zeta))$ hits $(I_n)$ for almost all $\zeta\in \partial\D$.
\end{thmD}
Note that condition \eqref{eq:muI1} implies that $\sum_{n=1}^\infty \mu_n=\infty$, so $(F_n)$ is contracting, and it also implies that $\sum_{n=1}^\infty |I_n|=\infty$.

%

In the proof of Theorem~D, we obtain a slightly more general sufficient condition than~\eqref{eq:muI1} which is somewhat more complicated to state and apply; see Theorem \ref{thm:contracting}. However, we will show that~\eqref{eq:muI1} is, at least when the sequence $(|I_n|)$ is decreasing, a best possible condition for determining how large a (shrinking) target needs to be in order to ensure that any given contracting forward composition of centred inner functions almost always hits the target; see Example~\ref{ex0.1}. 

\begin{rem}
Theorem~D will be deduced from Theorem~C by  arranging the functions $f_n$ in suitable consecutive blocks. Note, however, that Theorem~C is a special case of Theorem~D.
\end{rem}
 Given {\em any} contracting forward composition sequence $(F_n)$, we can find lengths $l_n$ with $\lim_{n\to\infty}l_n=0$ and $\sum_{n=1}^\infty \mu_n l_n=\infty$. Theorem~D then shows that $(F_n(\zeta))$ almost always hits any shrinking target $(I_n)$ such that $|I_n|=l_n$ for $n\in\N$. We immediately deduce the following corollary.
\begin{cor}\label{cor:denseorbits}
For $n\in\N$, let $f_n$ be centred inner functions such that $F_n = f_n \circ \cdots \circ f_1$ is contracting. Then the following equivalent properties hold:
\begin{itemize}
\item[(a)] for every $\zeta_0\in\partial\D$, we have $\liminf_{n\to\infty}|F_n(\zeta)-\zeta_0|= 0,$ for almost every $\zeta\in \partial \D$;
\item[(b)] for every arc $I\subset \partial\D$ of positive length, $F_n(\zeta)\in I$ infinitely often, for almost every $\zeta\in \partial \D$;
\item[(c)] $(F_n(\zeta))$ is dense in $\partial\D,$ for almost every $\zeta\in \partial \D$.
\end{itemize}
\end{cor}
 Here is a summary of our results concerning which forward compositions of centred inner functions hit shrinking targets.
\begin{summ}\label{summ-shrinking}
Let $F_n=f_n\circ \dots\circ f_1, n\in \N,$ where $f_n$ are centred inner functions, and $\lambda_n$ are defined as in \eqref{lambda-n} with $\mu_n=1-\lambda_n$.
\begin{itemize}
\item[(a)]
If $\sum_{n=1}^{\infty} |I_n| < \infty$, then $(F_n(\zeta)$ fails to hit $(I_n)$ for almost every $\zeta\in \partial \D$, by Theorem~A.
\item[(b)]
If $\sum_{n=1}^{\infty} \mu_n|I_n| = \infty$, then $(F_n)$ is contracting and $(F_n(\zeta)$) hits $(I_n)$ for almost every $\zeta\in \partial\D,$ by Theorem~D.
\item[(c)]
If $\sum_{n=1}^{\infty} |I_n| =\infty$ and $\sum_{n=1}^{\infty} \mu_n|I_n| < \infty$, then examples show that either of the conclusions in cases~(a) and~(b) is possible for contracting $(F_n)$.
\end{itemize}
\end{summ}

We now return to the dichotomy for iterates of inner functions due to Aaronson~\cite{aaronson}, and also to Doering and Ma\~{n}\'{e}~\cite{doering-mane}. One version of the dichotomy is the following.

\begin{thm}[ADM dichotomy]\label{thm:DM}
Let $f: \D \to \D$ be an inner function with a Denjoy--Wolff point $p\in\overline \D$.
\begin{enumerate}
\item[\rm(a)]If \;$\sum_{n\geq 0}(1-|f^n(0)|)<\infty$,
then $p\in\partial \D$ and $\lim_{n\to \infty}f^n(\zeta)=p$ for almost every $\zeta \in \partial \D$.
\item[\rm(b)] If \;$\sum_{n\geq 0}(1-|f^n(0)|)=\infty$, then $(f^n(\zeta))$ is dense in $\partial \D$ for almost every $\zeta \in \partial \D$.
\end{enumerate}
\end{thm}

We observe that part~(a) of the ADM dichotomy  was generalised in \cite[part~(a) of Theorem~4.1]{BMS} from iterates of inner functions to iterates of holomorphic self-maps of $\partial \D$, and  then generalised even further in \cite[Theorems A and B]{BEFRS2}; in particular, for arbitrary sequences of holomorphic self-maps of~$\D$, we have the following result.

\begin{thmE} \label{thm:intro 2}
Let $F_n, n\in \N,$  be a sequence of  holomorphic self-maps of $\mathbb{D}$ and suppose that
\begin{equation}\label{ThmB:ineqIntro}
\sum_{n=0}^{\infty} \left(1-|F_n(0)|\right) <\infty.
\end{equation}
Then  for almost all $\zeta\in \partial \D$ we have
\begin{equation}\label{eq:DW-prop}
{\rm dist}\,(F_n(\zeta), F_n(0))\to 0\;\text{ as } n\to \infty.
\end{equation}

\end{thmE}

\begin{rem}\label{BEFRS3.3}
Note that in the condition \eqref{ThmB:ineqIntro} the orbit $(F_n(0)$ can be replaced by any orbit $(F_n(z_0))$, where $z_0\in\D$. This follows from \cite[Theorem~3.3]{BEFRS2}, which shows that if one orbit converges to $\partial\D$, then all orbits converge to the boundary together at the same rate.
\end{rem}

In~\cite{BEFRS2}, we introduced the name {\em Denjoy--Wolff set} for the set of points $\zeta\in\partial \D$ that satisfy \eqref{eq:DW-prop}; with this terminology,  Theorem~E states that the Denjoy--Wolff set has full measure.

 Theorem~E can be deduced directly from Theorem~A, and we give the details  in Section~\ref{ThmEproof} for the reader's convenience and to show the connection between the failure to hit certain shrinking targets, when the self-maps of $\D$ are centred, and non-recurrent behaviour, when the orbits tend to $\partial \D$ quickly.
%

The situation is  less clear when it comes to part~(b) of Theorem~\ref{thm:DM}. In \cite{BEFRS2} we showed by example that an analogue of part~(b) cannot hold in general for sequences of holomophic self-maps of~$\D$, even for sequences of forward compositions of inner functions, but it may hold in some cases.


For iterates of an inner function, if the hypotheses of Theorem~\ref{thm:DM}, part~(b) hold, then the sequence of iterates $(f^n)$ must be contracting; see {\cite[Lemma~2.6, part~(b)]{BMS} or} the discussion in \cite[Section~8]{BEFRS2}, for example. So it seemed plausible to ask in \cite{BEFRS2} whether an analogue of Theorem~\ref{thm:DM}, part~(b) might hold for contracting forward compositions of inner functions. In fact, the situation turns out to be more complicated than expected and there are even examples of contracting forward compositions of inner functions satisfying $\sum_{n=1}^\infty(1-|F_n(0)|)=\infty$ for which the conclusion of Theorem~E holds (see Example~\ref{ex0.3}), thus answering \cite[Question 10.1]{BEFRS2} in the negative.

On the other hand, by using Theorem~B we can show that recurrent behaviour on $\partial \D$ does occur if the rates of contraction are sufficiently large compared to the sequence $(1-|F_n(0)|)$ and so obtain Theorem~F, a version of Theorem~\ref{thm:DM}, part~(b). 

%
%

\begin{thmF}
Let $F_n=f_n\circ \dots\circ f_1, n\in \N,$ where $f_n$ are inner functions. Let $\lambda_n$ be defined as in \eqref{lambda-n} and put $\mu_n=1-\lambda_n$. If
\begin{equation}\label{eq:muF1}
\sum_{n=1}^\infty \mu_n (1-|F_n(0)|)=\infty,
\end{equation}
 then $(F_n(\zeta))$ is dense in $\partial \D$ for almost every $\zeta\in\partial \D$, and   the Denjoy--Wolff set of $(F_n)$ has measure~0.
\end{thmF}
Once again, note that condition \eqref{eq:muF1} implies that $\sum_{n=1}^\infty (1-|F_n(0)|)=\infty$ and also that $\sum_{n=1}^\infty \mu_n=\infty$, so $(F_n)$ is contracting.

\begin{rem} 
 In   both Theorem E and Theorem~F, the orbit $(F_n(0))$ in $\D$ may tend to the boundary, and it may do so in such a way that it accumulates at every point of $\partial \D$, in which case almost all boundary orbits might be dense under $(F_n)$ just because they follow interior orbits.
To see the difference between the two situations, we can normalise the sequence $(F_n)$ by introducing rotations $(R_n)$ mapping $F_n(0)$ to $|F_n(0)|$ and define the related forward composition sequence
\[
G_n=g_n\circ\dots\circ g_1,\quad \text{where}\quad g_n=R_n \circ f_n \circ R_{n-1}^{-1},
\]
 so that  $(G_n(0))$ lies on the positive real axis and $\lim_{n\to\infty}G_n(0)= 1$.  Note that if $(F_n)$ satisfies \eqref{ThmB:ineqIntro} in Theorem E, so does $(G_n)$, and  we obtain that  $\lim_{n\to\infty}G_n(\zeta)= 1$ for almost every $\zeta\in\partial\D$, which in this case implies that  the Denjoy--Wolff set of $(G_n)$ has full measure. However, if $(F_n)$ satisfies \eqref{eq:muF1} in Theorem F, and hence $(G_n)$ does too, it follows that almost all orbits in $\partial \D$ are dense under $(G_n)$;  hence, in this case, the Denjoy--Wolff set of $(G_n)$ has measure~0. 
\end{rem}

\begin{rem}
 The summability  conditions in Theorem~E and Theorem~F are expressed in terms of  the orbit of~0 under $F_n$, but  they are
in fact independent of the initial point, with $\lambda_n=1-\mu_n$ defined to be a function of the initial point, say $z_0\in\D$, rather than~0. This holds because:
\begin{itemize}
\item by the Schwarz-Pick lemma, $F_n(z_0)$ lies in a hyperbolic disc with centre $F_n(0)$ and radius $\dist_{\D}(0,z_0)$, the hyperbolic distance from~0 to $z_0$; 
\item by \cite[Theorem~11.2]{BeardonMinda}, the hyperbolic distortion $\lambda_n(z_0)$ lies in a hyperbolic disc with centre $\lambda_n(0)$ and radius $2\dist_{\D}(0,z_0)$.
\end{itemize}
\end{rem}

We will show that \eqref{eq:muF1} is in fact a sharp condition on $(\mu_n)$ for ensuring that if $(F_n)$ is any contracting forward composition of inner functions such that $\sum_{n=1}^{\infty} (1-|F_n(0)|) = \infty$, then almost all boundary orbits for $(F_n)$ are dense in $\partial \D,$ as stated in Theorem F; 
see Example~\ref{ex0.2} and Example~\ref{ex0.3}. The latter example has the properties that
\begin{equation}\label{eq:three}
\sum_{n=1}^{\infty} (1-|F_n(0)|) = \infty,\quad \sum_{n=1}^{\infty}\mu_n (1-|F_n(0)|) < \infty, \quad \sum_{n=1}^{\infty} \mu_n = \infty,
\end{equation}
and $F_n(z)\to 1$ as $n\to\infty$, for all $z\in\D$ and almost all $z\in\partial\D$; that is, the Denjoy--Wolff set of $(F_n)$ has full measure,  the opposite conclusion to that of Theorem~F.

On the other hand, under the same conditions as in \eqref{eq:three}, there exist examples of sequences of forward compositions of inner functions such that almost all boundary orbits are dense for  $(F_n)$; see Example~\ref{innerfn}.

 Here is a summary of our extension of the ADM dichotomy to forward compositions of inner functions.
\begin{summ}\label{sum:DW-rec}Let $F_n=f_n\circ \dots\circ f_1, n\in \N,$ where $f_n, n\in \N,$ are inner functions, and $\lambda_n$ are defined as in \eqref{lambda-n} with $\mu_n=1-\lambda_n$.
\begin{itemize}
\item[(a)]
If $\sum_{n=1}^{\infty} (1-|F_n(0)|) < \infty$, then the Denjoy--Wolff set of $(F_n)$ has full measure in $\partial \D$, by Theorem~E.
\item[(b)]
If $\sum_{n=1}^{\infty} \mu_n(
1-|F_n(0)|) = \infty$, then $(F_n)$ is contracting,  almost all boundary orbits for $(F_n)$ are dense in $\partial \D$,  
and the Denjoy--Wolff set of $(F_n)$ has measure~0, by Theorem~F.
\item[(c)]
If $\sum_{n=1}^{\infty} (1-|F_n(0)|) = \infty$ and $\sum_{n=1}^{\infty} \mu_n(1-|F_n(0)|) < \infty$, then  examples show that either of the conclusions in cases~(a) and~(b) is possible for contracting $(F_n)$.
\end{itemize}
\end{summ}

\begin{rem}
We expect some of the results for forward compositions of inner functions in this paper to  have analogues for forward compositions of holomorphic maps between general simply connected domains, which is the setting used in \cite{BEFRS2}, and in particular for simply connected wandering domains of transcendental entire functions. For simplicity, we have restricted the treatment here to inner functions.
\end{rem}
Finally, our results suggest that the following question should be investigated.  For contracting forward compositions of inner functions, normalised so that the orbit of 0 lies on the positive real axis and converges to~1, is it the case that exactly one of these two cases must occur:
\begin{itemize}
\item[(a)] $F_n(\zeta)\to 1$ as $n\to \infty$ for almost all $\zeta\in\partial \D;$
\item[(b)] 
$(F_n(\zeta))$ is dense in $\partial \D$ for almost all $\zeta\in\partial \D$?
\end{itemize}
We remark that if either case~(a) or case~(b) occurs for all $\zeta$ in a subset of $\partial \D$ of positive measure, then the same case must hold for almost all $\zeta\in \D$; this follows easily from Lemma~\ref{lem:Pomgen}.

 As this paper was being finalised, we learnt of recent interesting work by Ferreira and Nicolau on forward composition sequences of centred inner functions in which they define and characterise the concept of `ergodicity' for such sequences (see~\cite{FerreiraNic}); this concept may have a role to play in investigating the above question. 

The structure of the paper is as follows. In Section~\ref{Prelims} we state known results needed in our proofs and in Section~\ref{TheoremA}, we prove Theorems~A,~B and~C. Then, in Section~\ref{nonunifresult}, we prove a result that includes Theorem~D. In Section~\ref{ThmEproof}, we prove Theorems~E and~F, and in Section~\ref{Examples} we give examples to show that our results are sharp.

\section{Preliminary lemmas}\label{Prelims}
To prove Theorems~A,~B and~C, we need several key lemmas. We start with two Borel--Cantelli lemmas, the first of which is entirely standard.

\begin{lem}[First Borel--Cantelli lemma]
Let $E_n$ be a sequence of measurable subsets of $\partial \D$, such that $\sum_{n=1}^\infty |E_n|<\infty$. Then the set $E$ of points that belong to $E_n$ for infinitely many~$n$ satisfies
\[
|E|= \left| \bigcap_{N=1}^\infty  \bigcup_{n=N}^\infty E_n\right| =0.
\]
\end{lem}
In the standard second Borel--Cantelli lemma, $\sum_{n=1}^\infty |E_n|=\infty$ and the probabilistic assumption of independence is required to show that almost all points lie in infinitely many of the sets~$E_n$. In our context, we use a version due to Philipp; see \cite[Theorem 3]{Philipp}.
\begin{lem}[Second Borel--Cantelli Lemma]\label{lem:Philipp}
Let $E_n$ be measurable sets in $[0,1]$. Let $A(N,x)$ be the number of $n \leq N$ such that $x \in E_n$ and let $\phi(N)= \sum_{n \leq N} |E_n|.$ Suppose there exist $c_k>0,$ with $\sum c_k < \infty,$ such that for all $n>m$
\[|E_n \cap E_m| \leq |E_n| |E_m| + |E_n| c_{n-m}.\]
Then, for $\varepsilon>0$ and almost every $x\in X$,
\[A(N,x) = \phi(N)+ O(\phi(N)^{1/2} \log \phi(N)^{3/2+ \varepsilon})\;\text{ as }N \to \infty.\]
In particular, if $\sum_{n=1}^\infty|E_n|=\infty$, then $A(N,x)\to\infty$ as $N\to\infty$ for almost every $x\in X$.
\end{lem}
In relation to inner functions, we need the following ergodic theory result of Pommerenke.
\begin{lem}[Pommerenke contracting lemma]\label{lem:Pomgen}
Let $(F_n)$ be a sequence of inner functions and suppose that $(F_n)$ is contracting.

If there are measurable subsets $L$ and $L_n$, $n\in\N$, of $\partial \D$, such that $L= F_n^{-1}(L_n)$, for $n\in\N$, up to a set of measure~0, then $L$ has either full or zero measure with respect to $\partial\D$.
\end{lem}
Lemma~\ref{lem:Pomgen} follows from a strong mixing theorem due to Pommerenke; see \cite[Theorem~1 and its discussion]{pommerenke-ergodic} and also \cite[Theorem~7.4]{BEFRS2}.

We need another mixing result, also given by Pommerenke \cite[Lemma 3]{pommerenke-ergodic}, concerning centred inner functions that satisfy a uniform contraction condition.
\begin{lem}[Pommerenke uniform contracting lemma] \label{lem:Pom3}
Let $f_n$ be centred inner functions with $|f_n'(0)| \leq \lambda<1$ for all $n \in \N$, and let $F_n = f_n \circ \cdots \circ f_1$. Then, for some absolute $K>0$ we have that
\[ \left| \frac{|A \cap F_n^{-1}(E)|}{|E|}- \frac{|A|}{2\pi}\right| \leq K \exp \left(- \frac{1-\lambda}{84}n\right),\quad\text{for }n\in \N,  \]
for all arcs $A$ and measurable sets $E$ in $\partial \D$ with $|E|>0$.
\end{lem}
We remark that it is possible (and useful) to allow the sets $E$ and $A$ in Lemma~\ref{lem:Pom3} to vary depending on~$n$ when applying Lemma~\ref{lem:Pom3}.

Finally, we need L\"owner's lemma, which can be found in \cite[Proposition 4.15]{Pommerenke} and \cite[Theorem 7.1.8 and Proposition 7.1.4 part~(4)]{BracciContrerasDM}. The case of equality for inner functions can be found in \cite[Corollary~1.5(b)]{doering-mane}, for example. Here we denote the harmonic measure of a Borel set~$A$ in the boundary of a domain~$U$ by $\omega(z, A, U)$.
\begin{lem}[L\"owner's lemma] \label{lem:Lowner}
Let~$f$ be a holomorphic self-map of~$\D$  and let $S\subset \partial \D$ be a Borel set. Then
\begin{equation}\label{eqtn:Ransford Phil}
\omega(z, f^{-1}(S), \D)\leq \omega(f(z), S,\D),\;\text{ for } z\in\D,
\end{equation}
with equality if $f$ is an inner function.
\end{lem}

\section{Proofs of Theorems~A,~B and~C}\label{TheoremA}
We start by proving Theorem~A, which states that if $(F_n)$  is a sequence of holomorphic self-maps of $\mathbb{D}$ such that $F_n(0)=0$ for $n\in \N$ and $(I_n)$ is a shrinking target in $\partial\D$ such that
\[
\sum_{n=1}^\infty |I_n|<\infty,
\]
then $(F_n(\zeta))$ fails to hit $(I_n)$ for almost all $\zeta\in\partial\D$. This is a straightforward consequence of L\"owner's lemma and the first Borel--Cantelli lemma.
\begin{proof}[Proof of Theorem~A]
By L\"owner's lemma, we have $|F_n^{-1}(I_n)|\le|I_n|,$ for $n\in \N$, so
\[
\sum_{n=1}^\infty|F_n^{-1}(I_n)|\le\sum_{n=1}^\infty|I_n|<\infty.
\]
Thus the set of points that lie in infinitely many of the sets $E_n:=F_n^{-1}(I_n)$ has measure~0 by the first Borel--Cantelli lemma.
\end{proof}
Next, Theorem~B states that if $F_n=f_n\circ \cdots \circ f_1$, $n\in \N$, is a contracting forward composition of centred inner functions and $(I_n)$ is any shrinking target, then \[
\{\zeta\in\partial \D: (F_n(\zeta)) \text{ hits } (I_n)\}
\]
has full or zero measure in $\partial \D$.
\begin{proof}[Proof of Theorem~B]
First, define
\[
L=\{\zeta\in \partial \D: (F_n(\zeta)) \text{ hits } (I_n)\},
\]
and, for each $m\ge 1$,
\begin{equation}\label{eq:Pom1}
L_m:= \{\zeta \in \partial \D : (f_{m+n} \circ \cdots \circ f_{m+1}(\zeta)) \text{ hits } (I_{m+n})\}.
\end{equation}
Then, by definition,
\begin{equation} \label{eq:Pom2}
 L = F_m^{-1}(L_m),\;\text{ for } m\in\N.
\end{equation}
So, since $(F_n)$ is contracting, we can apply Lemma~\ref{lem:Pomgen} to deduce that~$L$ has either full or zero measure in~$\partial\D$, as required.
\end{proof}
Theorem~C states that if $f_n$ are centred  inner functions with $|f_n'(0)| \leq \lambda<1$ for $n \in \N$, $F_n = f_n \circ \cdots \circ f_1$, and $(I_n)$ is a shrinking target that satisfies
\begin{equation}\label{eq:Indiv}
\sum _{n=1}^{\infty} |I_n|= \infty,
\end{equation}
then the sequence $(F_n(\zeta))$ hits $(I_n)$ for almost every $\zeta\in\partial \D$.
\begin{proof}[Proof of Theorem~C]
For $n\in \N$, put $E_n= F_n^{-1}(I_n)$.  Then
\[
\{\zeta \in \partial \D: F_n(\zeta) \in I_n \; \text{infinitely often}\}= \{\zeta \in \partial \D:\zeta \in E_n \; \text{infinitely often}\}= \bigcap_{N\geq 1}\bigcup_{n \geq N} E_n.
\]
We must show that this set has full measure.

By Lemma~\ref{lem:Lowner} we have $\omega (0, E_n, \D)= \omega(0,I_n, \D)$, for $n\in \N$. Thus
\begin{equation}\label{eq:EnIn}
|E_n|=|I_n|, \text{ for } n\in \N,\quad \text{so}\quad \sum_{n=1}^\infty |E_n|= \infty,
\end{equation}
by \eqref{eq:Indiv}. Since $\partial \D$ has Lebesgue measure $2\pi$, it suffices to prove that
\begin{equation}\label{eq:Thm1}
\frac{1}{2\pi}|E_n \cap E_m| \leq \left( \frac{1}{2\pi}|E_n| \right)\left( \frac{1}{2\pi}|E_m|\right) + |E_n|\,c_{n-m},
\end{equation}
where $\sum_{p=1}^{\infty} c_p <\infty$. Indeed, Lemma~\ref{lem:Philipp} will then imply that almost every $\zeta \in \partial \D$ lies in $E_n$ infinitely often, and the result follows.

Let $n=m+p$, where $p\ge 1$. Then, since $F_{m+p}=f_{m+p}\circ \cdots \circ f_{m+1}\circ F_m$, we have
\[
E_m \cap E_{m+p} = F_m^{-1}(I_m \cap (f_{m+1}^{-1} \circ \cdots \circ f_{m+p}^{-1} (I_{m+p}))).
\]
Thus,
\begin{align}
|E_m \cap E_{m+p}|&= 2\pi\omega(0, E_m \cap E_{m+p}, \D) \nonumber \\
&= 2\pi\omega(0,  I_m \cap (f_{m+1}^{-1} \circ \cdots \circ f_{m+p}^{-1} (I_{m+p})), \D) \nonumber \\
&= | I_m \cap (f_{m+1}^{-1} \circ \cdots \circ f_{m+p}^{-1} (I_{m+p}))|, \label{eq:hm1}
\end{align}
by Lemma~\ref{lem:Lowner}.

We now apply Lemma \ref{lem:Pom3} with $A=I_m$, $E=I_{m+p}$, $c=(1-\lambda)/84$ and the $p$ inner functions $f_{m+p}, \dots, f_{m+1}$. In this application,~$m$ is fixed and~$p$ is the running index. We obtain
\[
\left| \frac{ |I_m \cap (f_{m+1}^{-1} \circ \cdots \circ f_{m+p}^{-1} (I_{m+p}))|}{|I_{m+p}|} - \frac{|I_m|}{2\pi}\right| \leq K e^{-cp},
\]
for all $m, p \in \N$. Hence
\begin{equation}
\label{eq:pom}
|I_m \cap (f_{m+1}^{-1} \circ \cdots \circ f_{m+p}^{-1} (I_{m+p}))| \leq \frac{1}{2\pi} |I_m| |I_{m+p}|+ |I_{m+p}|Ke^{-cp}.
\end{equation}
Combining \eqref{eq:EnIn}, (\ref{eq:hm1}) and (\ref{eq:pom})  gives
\[
|E_m \cap E_{m+p}| \leq \frac{1}{2\pi}|E_m| |E_{m+p}|+ |E_{m+p}|Ke^{-cp},
\]
which is \eqref{eq:Thm1}. The result now follows from Lemma~\ref{lem:Philipp}.
\end{proof}

\section{Proof of Theorem~D}\label{nonunifresult}

 In this section, we use Theorem~C to prove Theorem~D, which is part~(b) of the following more general result. This concerns non-uniform contraction of a forward composition of inner functions and relates in various ways the size of the contraction to the size of the shrinking target that a given forward composition will hit.

\begin{thm} \label{thm:contracting}
For $n\in \N$, let $f_n$ be centred inner functions with  $|f_n'(0)| =\lambda_n=1-\mu _n$, let $F_n = f_n \circ \cdots \circ f_1$, and let $I_n \subset \partial \D.$
\begin{itemize}
\item[(a)]
Suppose that $\sum _{n=1}^{\infty} \mu_n= \infty$. Let
$n_k$ to be the least positive integer such that
\begin{equation}\label{eq:mu}
\mu_{1}+ \dots + \mu_{n_k} \in [k,k+1), \quad\text{for }k \geq 1.
\end{equation}
 Set $m_0=0$, and for $k\geq 1$ choose $m_k$ with  $n_k\le m_k < n_{k+1},$   such that
\begin{equation}\label{eq:max}
|I_{m_k}|=\max\{|I_n|: n_k\le n<n_{k+1}\}.
\end{equation}
If \;$\sum_{k=1}^\infty |I_{m_k}| =\infty$, then $(F_n(\zeta))$ hits $(I_n)$ for almost every $\zeta\in\partial\D$.
\item[(b)] If\; $\sum_{n=1}^\infty \mu_n|I_n|=\infty,$ then $(F_n(\zeta))$ hits $(I_n)$ for almost every $\zeta\in\partial\D$.
\item[(c)]
Suppose that $|I_n|\le c<1$ for $n\in \N$ and $\sum _{n=1}^{\infty} |I_n|= \infty$. Let
$n_k$ be the least positive integer such that
\begin{equation}\label{eq:I}
|I_{1}|+ \dots + |I_{n_k}| \in [k,k+1), \quad\text{for }k \geq 1,
\end{equation}
and choose $m_k$, with $n_k\le m_k<n_{k+1}-1,$ such that
\begin{equation}\label{eq:min}
\mu_{m_k}=\min\{\mu_n: n_k\le n<n_{k+1}-1\}.
\end{equation}
If \;$\sum_{k=1}^\infty \mu_{m_k} =\infty$, then $(F_n(\zeta))$ hits $(I_n)$ for almost every $\zeta\in\partial\D$.
\end{itemize}
\end{thm}

\begin{proof}
We first prove part~(a) and then show that  part~(a) implies part~(b), which in turn implies part~(c).

In part (a), we can assume that at least one of the series $\sum_{k=1}^{\infty} |I_{m_{2k}}|$, $\sum_{k=1}^{\infty} |I_{m_{2k+1}}|$ is divergent. The argument is similar in both cases and here we assume the first series is divergent.

Since, for every $k\geq 1$,
\[
n_{2k} \leq m_{2k} < n_{2k+1},
\]
it follows from the choice of $n_k$ in \eqref{eq:mu} that
\[
\sum_{i=1}^{m_{2(k-1)}} \mu_i  \leq  \sum_{i=1}^{n_{2k-1}-1} \mu_i <  2k-1
\quad
\text{and}
\quad
\sum_{i=1}^{m_{2k}} \mu_i \geq  \sum_{i=1}^{n_{2k}} \mu_i  \geq 2k.
\]
Hence, by subtracting these two inequalities, we obtain
\[
\mu_{(m_{2k-2}+1)}+\dots+\mu_{m_{2k}} \ge 1.
\]
It follows that
\[
\lambda_{(m_{2k-2}+1)} \dots \lambda_{m_{2k}}= (1-\mu_{(m_{2k-2}+1)}) \cdots (1-\mu_{m_{2k}}) \leq e^{-(\mu_{(m_{2k-2}+1)}+\dots+\mu_{m_{2k}})}\leq e ^{-1}.
\]
If we now set, for $k \ge 1$,
\[g_k=f_{m_{2k}} \circ \cdots \circ f_{(m_{2k-2}+1)},\]
then we have $g_k(0)=0$ and $|g_k'(0)|= \lambda_{m_{2k-2}+1} \cdots \lambda_{m_{2k}} \leq e^{-1}.$ Thus the sequence
\[
G_k=g_k \circ \cdots \circ g_1,\quad k\ge 1,
\] 
is uniformly contracting. 

Since $\sum_{k=1}^{\infty} |I_{m_{2k}}|=\infty$, it follows from Theorem~A that $\{\zeta \in \partial \D: G_k(\zeta) \in I_{m_{2k}},\;\text{infinitely often}\}$ has full measure. Since $G_k=F_{m_{2k}}$, we deduce that
$\{\zeta \in \partial \D: F_{m_{2k}}(\zeta) \in I_{m_{2k}},\;\text{infinitely often}\}$ has full measure, which gives the result.

Part~(b) follows from part~(a) because
\[
\mu_n|I_n|\le \mu_n|I_{m_k}|,\quad \text{for } \text{$n_k\le n<n_{k+1}$},
\]
and, by \eqref{eq:mu},
\[
\mu_{n_k}+\cdots + \mu_{n_{k+1}-1} \le 2,\quad\text{for }k\ge 1,
\]
so, by \eqref{eq:max},
\[
\infty=\sum_{n=n_1}^\infty \mu_n|I_n| =\sum_{k=1}^\infty \sum_{n=n_k}^{n_{k+1}-1} \mu_n|I_n| \le \sum_{k=1}^\infty \left(\sum_{n=n_k}^{n_{k+1}-1} \mu_n\right)|I_{m_k}|\le 2\sum_{k=1}^\infty |I_{m_k}|,
\]

as required.

Finally we show that part~(c) follows from part~(b). Indeed, since $|I_n|\le c<1$ for $n\in \N$, we have, by the choice of $n_k$ in \eqref{eq:I}, that
\[
\sum_{n=n_k}^{n_{k+1}-1} |I_n|\ge 1-c,\quad\text{for } k\ge 1.
\]
Thus, by \eqref{eq:min},
\[
\infty=\sum_{k=1}^\infty \mu_{m_k}\le \frac{1}{1-c} \sum_{k=1}^\infty\mu_{m_k} \left(\sum_{n=n_k}^{n_{k+1}-1} |I_n|\right)\le \frac{1}{1-c} \sum_{k=1}^\infty\sum_{n=n_k}^{n_{k+1}-1} \mu_n|I_n|
= \frac{1}{1-c}\sum_{n=n_1}^\infty \mu_n|I_n|,
\]
as required.
\end{proof}

\section{Proofs of Theorems~E and~F}\label{ThmEproof}

To start with, we prove the following theorem which gives a connection between  the property that  almost every boundary orbit is  dense and the property of hitting a shrinking target. It shows that almost every boundary orbit being dense in $\partial \D$ for a given  forward composition $(F_n)$ of inner functions is equivalent to the fact that a related forward composition of centred inner functions $(G_n)$ almost always hits certain shrinking targets on~$\partial \D$.

\begin{thm}\label{recurrenceshrinking}
Let $f_n$ be inner functions and $F_n = f_n \circ \cdots \circ f_1$ for $n\in \N$, and let $\lambda_n$ be the hyperbolic distortion of $f_n$, as defined in (\ref{lambda-n}). Let also
\begin{equation}\label{eq:Mobius}
M_n(z)= \frac{z+F_n(0)}{1+\overline{F_n(0)}z},\quad\text{for } n \in\N,
\end{equation}
and $M_0(z)=z$, and put $g_n= M_n^{-1} \circ f_n \circ M_{n-1}$, $n\in \N$, and $G_n = g_n \circ \cdots \circ g_1$. Then
\begin{itemize}
\item[(a)] $g_n(0)=0$ for all $n$;
\item[(b)] for each $n\in \N$, $|g'_n(0)|=\lambda_n$,  so $(G_n)$ is contracting if and only if $(F_n)$ has the same property; 
\item[(c)] almost every boundary orbit is dense for $(F_n)$ if and only if $(G_n(\zeta))$ almost always hits every shrinking target of arcs of the form $I_n = M_n^{-1}(I)$, $n \in\N$, for every arc $I\subset\partial \D$;
\item[(d)] with $I$ and $I_n$ as in part~(c), there exists a constant $c(I)>0$ such that 
\[
|I_n|\geq  c(I)(1-|F_n(0)|),\quad\text{for } n\in\N.
\]
In particular, 
if $(G_n(\zeta))$ almost always hits every shrinking target $(J_n)$ with size $|J_n|=c(1-|F_n(0)|)$, where $c>0$, then almost every boundary orbit $(F_n(\zeta))$  is dense in $\partial \D$.
\end{itemize}
\end{thm}
\begin{proof}
Part~(a) is straightforward, as is part~(b), since
\begin{align*}
|g_n'(0)|&=\left|(M_n^{-1})'(F_n(0))f_n'(F_{n-1}(0))M_{n-1}'(0)\right|\\
&=\frac{\rho_\D(0)}{|M_n'(0)|}  |f_n'(F_{n-1}(0))|   \frac{|M_{n-1}'(0)|}{\rho_\D(0)}\\
&=\frac{\rho_\D(F_{n}(0)) |f_n'(F_{n-1}(0))|}{\rho_\D(F_{n-1}(0))} =\lambda_n.
\end{align*}
Part~(c) is also straightforward,  in view of the equivalence of the properties in parts~(b) and~(c) of Corollary~\ref{cor:denseorbits}, since, for all arcs $I\subset \partial \D$ and $\zeta\in \partial \D$,
\[
 F_n(\zeta)= (M_n \circ G_n)(\zeta) \in I \quad\text{if and only if}\quad G_n(\zeta)\in I_n=M_n^{-1}(I).
\]

Let us prove (d).
Consider an arc $I=(e^{i\theta},e^{i\phi} )$, where $\theta<\phi<\theta+2\pi$ and let  $I_n=M_n^{-1}(I)$ be the arcs $(e^{i\theta_n},e^{i\phi_n} )$, $n\in \N$ as in part~(c). We can estimate the size of $|I_n|$ by a direct computation  as 
\begin{equation}\label{eq:partd1}
 2\sin(\tfrac12|I_n|)= |e^{i\phi_n}-e^{i\theta_n}| = |M_n^{-1}(e^{i\phi})-M_n^{-1}(e^{i\theta})|=\frac{(1-|F_n(0)|^2)|e^{i\phi}-e^{i\theta}|}{|F_n(0)-e^{i\phi}||F_n(0)-e^{i\theta}|}.
\end{equation}
In particular,
\[
 |I_n|\ge \frac14(1-|F_n(0)|^2)|e^{i\phi}-e^{i\theta}|\ge c(I)(1-|F_n(0)|).
\]
So if $(G_n(\zeta))$ almost always hits every shrinking target $(J_n)$ with $|J_n|=c(1-|F_n(0)|)$, where $c>0$, we deduce that $(G_n(\zeta))$ almost always hits every shrinking target $(I_n)$ of the form $I_n=M_n^{-1}(I), n\in\N$, for every arc $I\subset \partial\D$.
Hence  almost every boundary orbit $(F_n(\zeta))$ is dense by part~(c).
\end{proof}


We now show how Theorem~A can be applied to prove Theorem~E, which states that if $F_n, n\in \N,$  is a sequence of holomorphic self-maps of $\mathbb{D}$ and
\begin{equation}\label{ThmB:ineq}
\sum_{n=0}^{\infty} \left(1-|F_n(0)|\right) <\infty,
\end{equation}
then  for almost all $\zeta\in \partial \D$ we have
\[
{\rm dist}\,(F_n(\zeta), F_n(0))\to 0\;\text{ as } n\to \infty.
\]
\begin{proof}[Proof of Theorem~E]
Let $\zeta_n=e^{i \arg(F_n(0))}$, so ${\rm dist}(\zeta_n,F_n(0))=1-|F_n(0)|$.

Fix $r>0$ arbitrary and consider the arcs
\[
J_n(r):=D(\zeta_n,r) \cap \partial \D,\quad\emph{} n\in \N.
\]
Observe that $F_n(0)\in D(\zeta_n,r)$ for $n>N$  large enough, where $N=N(r)$.
Define the set
\[
E(r):=\{ \zeta\in\partial \D : F_n(\zeta)\in J_n(r) \text{\ for all $n$ sufficiently large}\}.
\]
We claim that $E(r)$ has full measure, independently of the value of~$r$.

Assuming that the claim is true, we choose a sequence $r_k\to 0$. Since $E(r)$ is decreasing in $r$, we have $E:=\bigcap_r E(r)=\bigcap_k E(r_k)$ is a countable intersection of full measure sets and therefore $E$ has full measure. On the other hand $E$ coincides with the set of $\zeta\in\partial \D$ such that $|F_n(\zeta)-F_n(0)|\to 0$ and we are done.

To see the claim, observe that
\[
E(r)=\{\zeta\in\partial \D: (F_n(\zeta)) \text{ does not hit $(J_n(r)^c)$} \},
\]
where $J_n(r)^c=\partial \D\setminus J_n(r)$.

Now, we recall the M\"obius transformations $M_n$ introduced in \eqref{eq:Mobius}, and define  $G_n: =M_n^{-1} \circ F_n$ for $n\in\N$, which are self-maps of $\D$ that fix the origin for all $n$. We deduce that
\[
E(r)=\{\zeta \in\partial \D: G_n(\zeta) \text{ does not hit } (I_n)\},\quad\text{where } I_n:=M_n^{-1}(J_n(r)^c), n\in \N.
\]
Now, each arc $I_n=M_n^{-1}(J_n(r)^c)$ lies in $\partial \D$ directly opposite the point $\zeta_n$, and it follows from \eqref{eq:partd1} that there exists a positive constant $C(r)$ such that $|I_n|\le  C(r) (1-|F_n(0)|)$, for $n\in \N$. Thus,  $\sum (1-|F_n(0)|) <\infty$ implies that $\sum |I_n|<\infty$ and the claim follows from Theorem~A applied to the centred functions $G_n$.
\end{proof}
We now show that Theorem~F is a fairly immediate consequence of Theorem~D and Theorem~\ref{recurrenceshrinking}. Theorem~F states that if $F_n=f_n\circ \dots\circ f_1,$ for $n\in \N,$ where $f_n$ are inner functions, $\lambda_n$ are defined as in \eqref{lambda-n}, $\mu_n=1-\lambda_n$, and
\begin{equation}\label{eq:muF2}
\sum_{n=1}^\infty \mu_n (1-|F_n(0)|)=\infty,
\end{equation}
then $(F_n(\zeta)$) is dense in  $\partial \D$ for almost all $\zeta\in \partial\D.$

\begin{proof}[Proof of Theorem~F]

Using the notation of Theorem~\ref{recurrenceshrinking}, we deduce from part~(b) of that theorem that the hyperbolic distortion of $f_n$ at $F_{n-1}(0)$ is $|g'_n(0)|$ and part~(c) shows that the density of almost all boundary orbits under $(F_n)$ corresponds under the sequence of M\"obius transformations $(M_n)$ to the property that $(G_n(\zeta))$ almost always hits the shrinking target. Part~(d) of Theorem~\ref{recurrenceshrinking} states that for every arc $I\subset\partial \D$, there is a constant $c(I)>0$ such that $|I_n|\ge c(I)(1-|F_n(0)|)$, so
\[
\sum_{n=1}^\infty \mu_n(1-|F_n(0)|)=\infty\quad \text{implies that}\quad \sum_{n=1}^\infty \mu_n|I_n|=\infty.
\]
Thus, we can apply Theorem~D to show that $(G_n(\zeta))$ hits $(I_n)$ for almost every $\zeta$ and hence  almost all boundary orbits for $(F_n)$ are dense in $\partial \D$
\end{proof}

\section{Examples}\label{Examples}
In this section we construct a number of examples to show the sharpness of our results.

First, Theorem~B states that if $F_n=f_n\circ \cdots \circ f_1$, $n\in \N$, is a contracting forward composition of centred inner functions and $(I_n)$ is a shrinking target, then
\[
\{\zeta\in\partial \D: (F_n(\zeta)) \text{ hits } (I_n)\}
\]
has either full measure or zero measure in $\partial \D$. Our first example shows that the hypothesis of contracting cannot be omitted from this theorem.
\begin{ex}[Sharpness of Theorem~B]\label{ex0.05}
There exits a shrinking target $(I_n)$ and a sequence of rotations $R_n$ such that
\begin{enumerate}[\rm (a)]
\item
$R_n(\zeta)$ hits $(I_n)$, for all $\zeta=e^{i\theta}$, $0\le \theta\le \pi$, but
\item
$R_n(\zeta)$ fails to hit $(I_n)$, for all $\zeta=e^{i\theta}$, $\pi < \theta < 2\pi$.
\end{enumerate}
\end{ex}
\begin{proof}
For $m\in\N$, let $I_m=\{e^{i\theta}: \pi\le \theta \le \pi+\pi/m\}$ and for $k=0,\ldots, m-1$ let
\[
R_{m,k}(z)=e^{\pi i (k+1)/m}z\quad \text{and}\quad I_{m,k}= \{e^{i\theta}: \pi -\pi (k+1)/m\le \theta \le \pi -\pi k/m\}.
\]
Observe that for any $m\in\N$, the sequence $(I_{m,k})_{0\leq k\leq m-1}$ divides the arc $[0,\pi]$ into $m$ subarcs. Then, for $m\in\N$,
\[
R_{m,k}(I_{m,k})=I_m, \quad \text{for } k=1,\ldots, m.
\]
Now we arrange the maps $R_{m,k}$ into a single sequence $R_n$, where  $n=\frac12m(m-1)+1+k$, $0\le k\le m-1$, $m\in\N.$ Then every $e^{i\theta}$ with $0\le \theta\le \pi$ belongs to an infinite number of arcs $I_{m,k}$, so there exists an infinite number of distinct values of $n$ for which $R_n(e^{i\theta})\in I_n$. Hence $R_n(e^{i\theta})$ hits the shrinking target $(I_n)$ for such $\theta$.

On the other hand it is clear that for $\pi < \theta < 2\pi$, $R_n(e^{i\theta})\not\in I_n$ for all $n\in\N,$ so $R_n(e^{i\theta})$ fails to hit $(I_n)$, for all $\pi \le \theta \le 2\pi$.
\end{proof}

Next we construct an example to show the sharpness of Theorem~D and Theorem~C, and then use this to construct examples to show the sharpness of Theorem~F. In particular, we give an example of a contracting forward composition of inner functions $F_n=f_n\circ \cdots \circ f_1$ which satisfies
\[
\sum_{n=0}^{\infty} \left(1-|F_n(0)|\right) =\infty
\]
and  for which boundary points with dense orbits do not form a full measure set. 
\begin{ex}[Sharpness of Theorem~D]\label{ex0.1}
Suppose that $(\mu_n)$ and $(l_n)$ are sequences, with $0<\mu_n,l_n\le \tfrac12$ for $n\in \N$, where $(l_n)$ is decreasing to~0,
\begin{equation}\label{mul-mul}
\sum_{n=1}^\infty \mu_n=\infty,\quad \sum_{n=1}^\infty l_n=\infty\quad\text{and}\quad \sum_{n=1}^\infty \mu_n l_n<\infty,
\end{equation}
and take $I_n$ to be the arc in $\partial \D$, centred at~1, with $|I_n| = l_n$ for $n\in \N$.

Then the forward composition sequence $B_n= b_n \circ \cdots \circ b_1$, where
\begin{equation}\label{eq:bnBn}
b_n(z)= z \frac{z+\lambda_n}{1+\lambda_nz},\quad\text{with } \lambda_n=1-\mu_n, n\ge1,
\end{equation}
fails to hit $(I_n)$ for almost every $\zeta\in \partial \D$.
\end{ex}

\begin{rem}
The hypotheses of Example~\ref{ex0.1} do not require that $\mu_n\to 0$ as $n\to \infty$; however, the conditions $\sum_{n=1}^\infty l_n=\infty$ and $\sum_{n=1}^\infty \mu_n \l_n<\infty$ together imply that $\liminf_{n\to \infty} \mu_n =0$.
\end{rem}

We shall see in Example~\ref{ex015} how  Example~\ref{ex0.1} can be used to show that Theorem~C is best possible since, if a positive sequence $(\mu_n)$ is given with $\mu_n\to 0$ as $n\to\infty$ and $\sum_{n=1}^\infty \mu_n=\infty$, then we can choose a decreasing sequence $(l_n)$ such that the other two conditions in \eqref{mul-mul} are satisfied.

\begin{proof}[Proof of Example~\ref{ex0.1}]
Our shrinking target is the nested sequence of arcs $(I_n)$ each with centre~1 such that $|I_n|=l_n$ for $n\in \N$. We need to show that the set
\[
 \{\zeta\in\partial\D : B_n(\zeta) \in I_n \text{\ infinitely often}\}
\]
has measure zero.

First, we note that $b_n(\pm1)= 1$, $b_n(0)=0$ and $b_n'(0)=\lambda_n$. Then denote by $R_n$ and $L_n$ the two analytic inverse branches of $b_n$, each defined on $\{z: |z-1|<\tfrac12\}$, such that $R_n(1)=1$ and $L_n(1)=-1$. Now define
\[
J_n:=R_n(I_n)\quad\text{and}\quad K_n:=L_n(I_n),\quad \text{for } n\in \N,
\]
which are arcs on $\partial \D$ centred at~1 and $-1$ respectively.

Now we estimate the behaviour of $b_n$ near the point~1. For $0<\theta\le\pi/3$, we have
\begin{align*}
\arg b_n(e^{i\theta})-\theta &=\arg\left(\frac{e^{i\theta}+\lambda_n}{1+\lambda_n e^{i\theta}}\right)\\
&=\arg(e^{i\theta}+\lambda_n)-\arg(1+\lambda_n e^{i\theta})\\
&=\tan^{-1}\left(\frac{\sin \theta}{\cos\theta+\lambda_n}\right)-\tan^{-1}\left(\frac{\lambda_n\sin \theta}{1+\lambda_n\cos\theta}\right)\\
&\le \frac{\sin \theta}{\cos\theta+\lambda_n} - \frac{\lambda_n\sin \theta}{1+\lambda_n\cos\theta}\quad\left(\text{since }\frac{\sin \theta}{\cos\theta+\lambda_n}>\frac{\lambda_n\sin \theta}{1+\lambda_n\cos\theta}\right)\\
&= \frac{\sin\theta(1-\lambda_n^2)}{(\cos\theta+\lambda_n)(1+\lambda_n\cos\theta)}\\
&\le 2\mu_n\sin\theta,
\end{align*}
since $\cos \theta \ge \cos \pi/3=\tfrac12$ and $\tfrac12\le \lambda_n< 1$.

Because $I_n$ is symmetric about the point~1 and $|I_n|=l_n\le \tfrac12$, and $R_n$ is the inverse branch of $b_n$ such that $R_n(1)=1$, we deduce that, for $n\in \N$,
\[
J_n=R_n(I_n)\subset I_n \quad\text{satisfies} \quad \tfrac12|I_n|-\tfrac12|J_n|\le 2\mu_n\sin(\tfrac12 |J_n|)\le \mu_n|I_n|.
\]
It follows by Lemma~\ref{lem:Lowner} that
\begin{equation}\label{eq:cvgt}
|K_n|=|I_n|-|J_n| \le 2\mu_n|I_n|.
\end{equation}
Next note that
\[
\{\zeta\in\partial\D: B_n(\zeta)\in I_n, \text{ infinitely often}\}\subset\bigcup_{n=N}^\infty B_n^{-1}(I_n),\;\;\text{for each }N\in\N.
\]
We now examine the structure of $B_n^{-1}(I_n)$ for $n\in \N$.

In general, for $n\ge 2$, $J_{n}=R_n(I_n)\subset I_n \subset I_{n-1}$, so
\begin{align*}
B_n^{-1}(I_n)&= B_{n-1}^{-1}(R_n(I_n)\cup L_n(I_n))\\
&\subset B_{n-1}^{-1}(I_{n-1} \cup K_n)\\
&=B_{n-1}^{-1}(I_{n-1}) \cup B_{n-1}^{-1}(K_n).
\end{align*}
Thus, for $N\in \N$,
\[
\bigcup_{n=N}^\infty B_n^{-1}(I_n) \subset \bigcup_{n=N}^\infty B_{n-1}^{-1}(K_n)\cup B_N^{-1}(I_N).
\]
By Lemma~\ref{lem:Lowner} we have $|B_{n-1}^{-1}(K_n)|=|K_n|$ and $|B_{N}^{-1}(I_N)|=|I_N|$. Thus, by \eqref{eq:cvgt},
\[
\left|\bigcup_{n=N}^\infty B_n^{-1}(I_n)\right|\le \sum_{n=N}^\infty |K_n|+|I_N|\le 2\sum_{n=N}^\infty \mu_n|I_n|+|I_N| =: \epsilon_N.
\]
Since $\epsilon_N\to 0$ as $N\to\infty$ by hypothesis, the  proof is complete.
\end{proof}

\begin{rem}
The proof that Example~\ref{ex0.1} has the required property relies on the fact that the nested arcs $(I_n)$ are related to the maps $B_n=b_n\circ \cdots \circ b_1$ in a particularly nice way; to be precise, for all $n\ge 2$, we have
\[
B_n^{-1}(I_n)\subset B_{n-1}^{-1}(I_{n-1}) \cup B_{n-1}^{-1}(K_n)
\]
and
\[
\sum_{n=1}^\infty |B_{n-1}^{-1}(K_n)|=\sum_{n=1}^\infty |K_n|<\infty.
\]
It seems likely that if the arcs are aligned differently, for example, if they are all centred at $-1$ rather than 1, then we can have $(B_n)$ hits $(I_n)$ for almost every $\zeta\in\partial\D$.
\end{rem}

\begin{rem}
More generally, if we have functions $(F_n)$ and a shrinking target $(I_n)$ such that, for some $m\in \N$ and all $n\ge m+1$,
\[
F_n^{-1}(I_n) \subset F_{n-1}^{-1}(I_{n-1})\cup \cdots \cup F_{n-m}^{-1}(I_{n-m}) \cup K_{n,m}
\]
and
\[
\sum_{n=m+1}^\infty |K_{n,m}|<\infty,
\]
then $(F_n)$ fails to hit $(I_n)$ for almost every $\zeta\in\partial\D$.
\end{rem}

Next we demonstrate the extent to which Theorem~C is best possible.
\begin{ex}[Sharpness of Theorem~C]\label{ex015}
Let $(\lambda_n)$ be a sequence in $(0,1)$ such that $\lambda_n\to 1^-$ as $n\to\infty$ and put $\mu_n=1-\lambda_n$ for $n\in \N$. Suppose that $\sum_{n=1}^\infty \mu_n=\infty$.

Then there is a decreasing positive sequence $(l_n)$ such that $\lim_{n\to\infty}l_n=0$, $\sum_{n=1}^\infty l_n=\infty$, and a shrinking target $(I_n)$ with $|I_n|=l_n$ for $n\in \N$ such that the forward composition sequence $(B_n)$ defined in \eqref{eq:bnBn} fails to hit $(I_n)$ for almost every $\zeta\in \partial \D$.
\end{ex}
\begin{proof}
We construct the sequence $(l_n)$ explicitly in terms of $\mu_n$ in such a way that the hypotheses of Example~\ref{ex0.1} hold.

First we choose a subsequence $(\mu_{n_k})$, where $1<n_1<n_2<\cdots,$ such that
\[
\mu_{n_k}\ge \mu_n,\;\text{ for } n\ge n_k, k\ge 1,
\]
\[
n_{k+1}-n_k \nearrow \infty\;\text{ as }k\to\infty,\quad\text{and}\quad \sum_{k=1}^\infty \mu_{n_k}<\infty.
\]
This choice is possible since $\lim_{n\to\infty}\mu_n=0$.

Now put
\[
l_n:=\frac{1}{n_{k+1}-n_k},\quad\text{for } n_k\le n\le n_{k+1}-1, k\ge 0,
\]
where we take $n_0=1$. Then $(l_n)$ is decreasing and $\lim_{n\to\infty}l_n=0$,
\[
\sum_{n=1}^\infty l_n=\sum_{k=1}^\infty \sum_{n=n_k}^{n_{k+1}-1} 1/(n_{k+1}-n_k)=\sum_{k=1}^\infty 1=\infty,
\]
and
\[
\sum_{n=1}^\infty \mu_nl_n=\sum_{k=1}^\infty \sum_{n=n_k}^{n_{k+1}-1} \mu_n/(n_{k+1}-n_k)\le\sum_{k=1}^\infty \mu_{n_k}<\infty,
\]
as required.
\end{proof}
\begin{rem}
Because of the symmetry of $\mu_n$ and $l_n$ in this result, we could alternatively assume that we are given a (decreasing) sequence $(l_n)$ with $\lim_{n\to\infty}l_n=0$ and $\sum_{n=1}^\infty l_n=\infty$, and obtain the sequence $(\mu_n)$ so that the hypotheses of Example~\ref{ex0.1} hold.
\end{rem}

We now construct the promised forward composition of inner functions, related to Theorem~F, which is contracting but for which not almost all  boundary orbits are dense in $\partial\D$.

\begin{ex}[Sharpness of Theorem~F]\label{ex0.2} 
Let $\mu_n$, $b_n$ and $I_n$ be defined as in Example~\ref{ex0.1}, satisfying \eqref{mul-mul}. Put
\begin{equation}\label{eq:fnbn}
F_n= f_n \circ \cdots \circ f_1,\quad\text{where}\quad f_n= M_{n} \circ b_n \circ M_{n-1}^{-1},
\end{equation}
and $M_n$ is the M\"obius map such that
$M_n(I_n)=I:=\{e^{it}: \pi/2 \leq t \le 3\pi/2\}, \;\text{for } n\in \N,$ and $M_0$ is the identity. Then, the hyperbolic distortion of $f_n$ at $F_{n-1}(0)$ is $\lambda_n=1-\mu_n$, so the sequence $(F_n)$ is contracting, and
\begin{equation}\label{Fnprops}
\sum_{n=1}^{\infty} (1-|F_n(0)|) = \infty,\;\; \sum_{n=1}^{\infty}\mu_n (1-|F_n(0)|) < \infty, \;\; F_n(z) \to 1 \text{ as } n\to\infty,\text{  for } z \in \mathbb{D},
\end{equation}
and yet $F_n(\zeta) \in I$ at most finitely often for almost every $\zeta\in\partial \D$,  so none of these boundary orbits can be dense in $\partial\D$.  
\end{ex}
\begin{proof} Note that, for $n\in \N$, we have $F_n=M_n\circ B_n$ and $F_n(0)=M_n(0)$ is real and positive.

The fact that the hyperbolic distortion of $f_n$ at $F_{n-1}(0)$ is $\lambda_n=1-\mu_n$ follows from \eqref{eq:bnBn} and \eqref{eq:fnbn}. Thus $(F_n)$ is contracting, by \eqref{mul-mul} and \eqref{eq:contracting}. The property that $F_n(\zeta) \in I$ at most finitely often for almost every $\zeta\in\partial \D$ follows from the conclusion of Example~\ref{ex0.1}.

Using the facts that $|I_n|\to 0$ as $n\to\infty$, $\sum_{n=1}^\infty|I_n|=\infty$ and $\sum_{n=1}^\infty\mu_n|I_n|<\infty$, we shall deduce that the three properties in \eqref{Fnprops} hold. To do this we use the relationship between the lengths of the arcs $I_n$ and the length of $M_n(I_n)=I$, obtained from \eqref{eq:partd1} in the proof of Theorem~\ref{recurrenceshrinking}:
\[
\sin(\tfrac12|I_n|)=\frac{1-|F_n(0)|^2}{|F_n(0)-i||F_n(0)+i|}.
\]
Since $|I_n|\to 0$ as $n\to\infty$ and  $F_n(0)$  lies in $(0,1)$, this implies that $F_n(0) \to 1$ as $n\to\infty$, so  $F_n(z) \to 1$  as $n\to\infty$ for  $z \in \mathbb{D}$,  by \cite[Theorem~3.3]{BEFRS2}; see Remark~\ref{BEFRS3.3}. Moreover,  $(1-|F_n(0)|)/\frac12|I_n|\to 1$ as $n\to\infty$, which gives
\[
\sum_{n=1}^\infty (1-|F_n(0)|)= \infty\quad\text{and}\quad \sum_{n=1}^\infty \mu_n(1-|F_n(0)|)<\infty,
\]
as required.
\end{proof}
By modifying Example~\ref{ex0.2} slightly we can even arrange that the orbit of almost every boundary point converges to $1$.
\begin{ex}\label{ex0.3} Let $b_n$ and $I_n$ be defined as in Example~\ref{ex0.1} and let $\tilde I_n$ denote the subarc of $I_n$ with centre~1 and length $|I_n|/t_n$, where $t_n=\sum_{k=1}^n |I_k|$. Then put
\[
\tilde F_n= \tilde f_n \circ \cdots \circ \tilde f_1,\quad\text{where}\quad \tilde f_n= \tilde M_{n} \circ b_n \circ \tilde M_{n-1}^{-1},
\]
where $\tilde M_n$ is the M\"obius map such that
$
\tilde M_n(\tilde I_n)=I:=\{e^{it}: \pi/2 \leq t \le 3\pi/2\}, \;\text{for } n\in \N,
$
and $\tilde M_0$ is the identity. Then
\[
\sum_{n=1}^{\infty} (1-|\tilde F_n(0)|) = \infty, \quad \tilde F_n(z) \to 1 \text{ as } n\to\infty,\text{  for all } z \in \mathbb{D},
\]
$(\tilde F_n)$ is contracting, and $\tilde F_n(\zeta)\to 1$ as $n\to \infty$ for almost every $\zeta\in \partial\D.$
\end{ex}
\begin{proof}
First note that since $\sum_{n=1}^\infty |I_n|=\infty$ it follows that $\sum_{n=1}^\infty |\tilde I_n|=\infty$, by a theorem of Abel; see \cite[item~162, page~120]{inequalities}, for example. Also, since $|I_n|/|\tilde I_n|=t_n\to \infty$ as $n\to \infty$ and $\tilde M_n(\tilde I_n)=I$, it follows that the sequence of arcs $\tilde M_n(I_n)$ expands to fill $\partial \D\setminus \{-1\}$.

Since $B_n(\zeta)\in I_n$ at most finitely often for almost every $\zeta \in \partial \D$, it follows that
\[
\tilde F_n(\zeta) =\tilde M_n\circ B_n\circ \tilde M_{0}^{-1}\in \tilde M_n(I_n),
\]
at most finitely often, for almost every $\zeta \in \partial \D$. The result follows.
\end{proof}

Our final example relates to part~(c) of Summary~\ref{sum:DW-rec}, our results related to the ADM dichotomy, by showing that sequences satisfying the same conditions as Example \ref{ex0.2}, stated in  (\ref{Fnprops}), may have the opposite behaviour.
\begin{ex}\label{innerfn}
There exists a forward composition of inner functions $(F_n)$ with contraction sequence $\lambda_n=1-\mu_n$ such that
\begin{equation}\label{eq:three1}
\sum_{n=1}^{\infty} (1-|F_n(0)|) = \infty,\quad \sum_{n=1}^{\infty}\mu_n (1-|F_n(0)|) < \infty, \quad \sum_{n=1}^{\infty} \mu_n = \infty,
\end{equation}
and almost all boundary orbits for $(F_n)$ are dense in $\partial \D.$
\end{ex}
\begin{proof}
We define $F_n=f^n$ for $n\in \N$, where $f$ is the Blaschke product
\[
f(z)=\left(\frac{z+1/3}{1+z/3}\right)^2,
\]
which is an inner function with Denjoy--Wolff point~1, a parabolic fixed point.

We now estimate $\mu_n$ and $\lambda_n$, which are not constant since $f$ does not fix the origin. Since $f(x)>0$ for $-1<x<1$, $f(1)=1$, $f'(1)=1$, $f''(1)=0$ and $f'''(1)=-1/16\neq 0$, we have
\begin{equation}\label{eq:est}
1-F_n(0)=1-f^n(0)\sim \frac1{n^{1/2}} \;\text{ as }n\to\infty;
\end{equation}
see, for example, \cite[proof of Theorem~3.4]{BEFRS}, which builds on the proof of \cite[Theorem~6.5.4]{Beardon}. Also,
\begin{equation}\label{eq:density}
\rho_{\D}(z)=\frac2{1-|z|^2}, \text{ for } z\in\D.
\end{equation}
Hence,
\begin{align*}
\lambda_n&=\frac{\rho_\D(f^{n}(0)) |f'(f^{n-1}(0))|}{\rho_\D(f^{n-1}(0))}\\
&=\frac{(1-f^{n-1}(0)^2)f'(f^{n-1}(0))}{1-f^{n}(0)^2}.
\end{align*}
Now a calculation using \eqref{eq:est} and the Taylor expansion of~$f$ about the point~1 shows that
\[
\mu_n=1-\lambda_n \sim \frac1n\;\text{ as } n\to \infty.
\]
Thus the conditions in \eqref{eq:three1} are satisfied. In particular, Theorem~\ref{thm:DM}, part~(b) can be applied to show that almost all boundary orbits for $(F_n)$ are dense in $\partial \D$.
\end{proof}


\bibliography{Wandering2}

\newcommand{\etalchar}[1]{$^{#1}$}
\begin{thebibliography}{BCDM20}

\bibitem[Aar78]{aaronson}
Jon Aaronson.
\newblock Ergodic theory for inner functions of the upper half plane.
\newblock {\em Ann. Inst. H. Poincar\'{e} Sect. B (N.S.)}, 14(3):233--253,
  1978.

\bibitem[AP19]{MagnusA}
Magnus Aspenberg and Tomas Persson.
\newblock Shrinking targets in parametrised families.
\newblock {\em Math. Proc. Cambridge Philos. Soc.}, 166(2):265--295, 2019.

\bibitem[BCDM20]{BracciContrerasDM}
Filippo Bracci, Manuel~D. Contreras, and Santiago D\'{\i}az-Madrigal.
\newblock {\em Continuous semigroups of holomorphic self-maps of the unit
  disc}.
\newblock Springer Monographs in Mathematics. Springer, Cham, 2020.

\bibitem[Bea91]{Beardon}
Alan~F. Beardon.
\newblock {\em Iteration of rational functions}, volume 132 of {\em Graduate
  Texts in Mathematics}.
\newblock Springer-Verlag, New York, 1991.
\newblock Complex analytic dynamical systems.

\bibitem[BEF{\etalchar{+}}22]{BEFRS}
Anna~Miriam Benini, Vasiliki Evdoridou, N\'uria Fagella, Philip Rippon, and
  Gwyneth Stallard.
\newblock Classifying simply connected wandering domains.
\newblock {\em Math. Ann.}, 383:1127--1178, 2022.

\bibitem[BEF{\etalchar{+}}24]{BEFRS2}
A.~M. Benini, V.~Evdoridou, N.~Fagella, P.~J. Rippon, and G.~M. Stallard.
\newblock Boundary dynamics for holomorphic sequences, non-autonomous dynamical
  systems and wandering domains.
\newblock {\em Advances in Mathematics}, 446, 2024.
\newblock https://doi.org/10.1016/j.aim.2024.109673.

\bibitem[BK24]{BakerKoi}
Simon Baker and Henna Koivusalo.
\newblock Quantitative recurrence and the shrinking target problem for
  overlapping iterated function systems.
\newblock {\em Adv. Math.}, 442:Paper No. 109538, 65, 2024.

\bibitem[BM07]{BeardonMinda}
A.~F. Beardon and D.~Minda.
\newblock The hyperbolic metric and geometric function theory.
\newblock In {\em Quasiconformal mappings and their applications}, pages 9--56.
  Narosa, New Delhi, 2007.

\bibitem[BMS05]{BMS}
P.~S. Bourdon, V.~Matache, and J.~H. Shapiro.
\newblock On convergence to the {D}enjoy--{W}olff point.
\newblock {\em Illinois J. Math.}, 49(2):405--430, 2005.

\bibitem[DM91]{doering-mane}
C.I. Doering and R.~{Ma\~{n}\'{e}}.
\newblock The dynamics of inner functions.
\newblock {\em Ensaios Matem\'aticos, Sociedade Brasileira de Matem\'atica},
  3:1--79, 1991.

\bibitem[FMP07]{Fern}
J.~L. Fern\'{a}ndez, M.~V. Meli\'{a}n, and D.~Pestana.
\newblock Quantitative mixing results and inner functions.
\newblock {\em Math. Ann.}, 337(1):233--251, 2007.

\bibitem[FN24]{FerreiraNic}
Gustavo~Rodrigues Ferreira and Artur Nicolau.
\newblock Mixing and ergodicity of compositions of inner functions.
\newblock {\em Preprint, arXiv:2405.06411v1}, 2024.

\bibitem[HLP52]{inequalities}
G.~H. Hardy, J.~E. Littlewood, and G.~P\'{o}lya.
\newblock {\em Inequalities}.
\newblock Cambridge, at the University Press,, 1952.
\newblock 2d ed.

\bibitem[HV95]{HV}
Richard Hill and Sanju~L. Velani.
\newblock The ergodic theory of shrinking targets.
\newblock {\em Invent. Math.}, 119(1):175--198, 1995.

\bibitem[IU23]{ivrii-urb}
Oleg Ivrii and Mariusz Urba\'nski.
\newblock Inner functions, composition operators, symbolic dynamics and
  thermodynamic formalism.
\newblock {\em Preprint, arXiv:\,2308.16063v1}, 2023.

\bibitem[KKP20]{Kirsebom}
Maxim Kirsebom, Philipp Kunde, and Tomas Persson.
\newblock Shrinking targets and eventually always hitting points for interval
  maps.
\newblock {\em Nonlinearity}, 33(2):892--914, 2020.

\bibitem[Phi67]{Philipp}
Walter Philipp.
\newblock Some metrical theorems in number theory.
\newblock {\em Pacific J. Math.}, 20:109--127, 1967.

\bibitem[Pom81]{pommerenke-ergodic}
Christian Pommerenke.
\newblock On ergodic properties of inner functions.
\newblock {\em Math. Ann.}, 256(1):43--50, 1981.

\bibitem[Pom92]{Pommerenke}
Ch. Pommerenke.
\newblock {\em Boundary behaviour of conformal maps}, volume 299 of {\em
  Grundlehren der Mathematischen Wissenschaften [Fundamental Principles of
  Mathematical Sciences]}.
\newblock Springer-Verlag, Berlin, 1992.

\end{thebibliography}
\end{document}